\documentclass{birkjour}
 \newtheorem{thm}{Theorem}[section]
 \newtheorem{cor}[thm]{Corollary}
 
 \newtheorem{prop}[thm]{Proposition}
 \theoremstyle{definition}
 
 \theoremstyle{remark}

 \numberwithin{equation}{section}

\newcommand{\Hol}{{\mathcal H}}
\newcommand{\D}{{\mathbb D}}
\newcommand{\C}{{\mathbb C}}
\newcommand{\w}{\omega}

\newcommand{\dom}{\mathcal D}
\newcommand{\ran}{\mathcal R}
\newcommand{\uP}{{\mathbb U}}
\newcommand{\R}{{\mathbb R}}

\newcommand{\Z}{{\mathbb Z}}

\newcommand{\al}{\alpha}
\newcommand{\la}{\lambda}
\newcommand{\g}{\gamma}
\newcommand{\G}{\Gamma}
\newcommand{\vj}{\varphi}
\newcommand{\s}{\sigma}
\newcommand{\z}{\zeta}

\newcommand{\Span}{\operatorname{span}}

\newcommand{\cl}{\operatorname{cl}}

\newcommand{\B}{{\mathcal B}}

\newcommand{\Aut}{\operatorname{Aut}}

\begin{document}

%-------------------------------------------------------------------------

%---------------------------------------------------------------------------

%

\title[Groups of Composition operators]
 {Weighted Composition Groups on the Little Bloch space}

%%----------Author 1
\author[S. B. Mose]{S. B. Mose}
\address{
Department of Pure and Applied Mathematics \\
Maseno University\\
P.O. Box 333 - 40105\\
Maseno\\
Kenya}

%%----------Author 2
\author[J. O. Bonyo]{J. O. Bonyo$^{1}$}
\address{
Department of Pure and Applied Mathematics \\
Maseno University\\
P.O. Box 333 - 40105\\
Maseno\\
Kenya}
\email{jobbonyo@maseno.ac.ke}



\dedicatory{Dedicated to Prof. Len Miller (PhD advisor to author$^1$) and Prof. Vivien Miller of Mississippi state University on their retirement}
%%% ----------------------------------------------------------------------

\begin{abstract}
We determine both the semigroup and spectral properties of a group of weighted composition operators on the Little Bloch space. It turns out that these are strongly continuous groups of invertible isometries on the Bloch space. We then obtain the norm and spectra of the infinitesimal generator as well as the resulting resolvents which are given as integral operators. As consequences, we complete the analysis of the adjoint composition group on the predual on the nonreflexive Bergman space, and a group of isometries associated with a specific automorphism of the upper half plane.
\end{abstract}

%%% ----------------------------------------------------------------------
\maketitle
%%% ----------------------------------------------------------------------
%\tableofcontents
\section{Introduction}
The (open) unit disc $\D$ of the complex plane $\C$ is defined as $\D = \{z\in \C : |z| < 1\}$, while the upper half-plane of $\C$, denoted by $\uP$, is given by $\uP=\{\w \in \C: \Im(\w)>0\}$ where $\Im(\w)$ stands for the imaginary part of $\w$. The Cayley transform $\psi(z):=\frac{i(1+z)}{1-z}$ maps the unit disc $\D$ conformally onto the upper half-plane $\uP$ with inverse $\psi^{-1}(\w)=\frac{\w-i}{\w+i}.$ For every $\al > -1$, we define a positive Borel measure $dm_{\alpha}$ on $\D$ by $dm_{\alpha}(z) = (1-|z|^{2})^{\al}dA(z)$ where $dA$ denotes the area measure on $\D.$\\
For an open subset $\Omega$ of $\C$, let $\Hol(\Omega)$ denote the Fr\'echet space of analytic functions $f:\Omega\to\C$ endowed
with the topology of uniform convergence on compact subsets of $\Omega$. Let $\Aut(\Omega) \subset \Hol(\Omega)$ denote the group of biholomorphic maps $f: \Omega \to \Omega$.  For $1 \leq p < \infty$, $\al> -1$, the weighted Bergman spaces of the unit disc $\D$, $ L_a^p(\D,m_\al)$, are defined by
\begin{equation*}
    L_a^p(\D,m_\al):=\left\{ f\in\Hol(\D): \|f\|_{L_a^p(\D,m_\al)}=\left(\int_{\D}|f(z)|^p\,dm_\al(z)\right)^{\frac{1}{p}}< \infty \right\}.
\end{equation*}
Clearly $ L_a^p(\D,m_\al) =  L^p(\D,m_\al)\cap \Hol(\D)$ where $ L^p(\D,m_\al)$ is the classical Lebesgue spaces.
For every $f\in  L_a^p(\D,m_\al)$, the growth condition is given by
\begin{equation*}
    |f(z)| \leq \frac{K\|f\|}{(1-|z|^2)^{\g}}
\end{equation*}
where $K$ is a constant and $\g = \frac{\al +2}{p}$, see for example \cite[Theorem 4.14]{Zhu}.\\
The Bloch space of the unit disc, denoted by $\B_{\infty}(\D)$, is defined as the space of analytic functions $f\in \Hol(\D)$ such that the seminorm
\begin{equation*}
    \|f\|_{\B_{\infty,1}(\D)}:= \sup_{z\in\D}\left(1-|z|^2\right)|f'(z)| < \infty.
\end{equation*}
Following \cite{Zhu2, Zhu}, $\B_{\infty}(\D)$ is a Banach space with respect to the norm $\|f\|_{\B_{\infty}(\D)}:=|f(0)| + \|f\|_{\B_{\infty,1}(\D)}$. On the other hand, the Little Bloch space of the disc, denoted by $\B_{\infty,0}(\D)$, is defined to be the closed subspace of $\B_{\infty}(\D)$ such that
\begin{equation*}
    \B_{\infty,0}(\D):= \cl_{\B_{\infty}}\C[z]
\end{equation*}
where $\cl_{\B_{\infty}}\C[z]$ denotes $\B_{\infty}$ closure of the set of analytic polynomials in $z$. Equivalently,
\begin{equation*}
    \B_{\infty,0}(\D):= \left\{ f\in \Hol(\D):\,\lim_{|z|\to 1^-,\,z\in \D}\left(1-|z|^2\right)|f'(z)|=0 \right\},
\end{equation*}
and possesses the same norm as $\B_{\infty}(\D).$ Since  $\B_{\infty,0}(\D)$ is a closed subspace of the Banach space $\B_{\infty}(\D)$, it follows that $\B_{\infty,0}(\D)$ is a Banach space as well with respect to the norm $\|\cdot\|_{\B_{\infty}(\D)}$. Note that every $f\in \B_{\infty}(\D)$ (or $f\in \B_{\infty,0}(\D)$) satisfies the growth condition
\begin{equation}\label{eq1}
    |f(z)| \leq \left(1+\tfrac1{2}\log \left(\frac{1+|z|}{1-|z|}\right) \right)\|f\|_{\B_{\infty}(\D)}.
\end{equation}
See for instance \cite{Ohno} for details.
Let $1 < p < \infty$ and $q$ be conjugate to $p$ in the sense that $\frac{1}{p} + \frac{1}{q} = 1$. If $(L_a^p(\D,m_\al))^*$ is the dual space of $L_a^p(\D,m_\al)$, then
\begin{equation}\label{Lpdual}
    (L_a^p(\D,m_\al))^* \approx L_a^q(\D, m_\al), \quad \al > -1,
\end{equation}
under the integral pairing
\begin{equation}\label{Lpdualp}
    \langle f,g\rangle=\int_\D f(z)\overline{g(z)}\,dm_\al\quad (f\in L^p_a(\D, m_\al),\ g\in L^q_a(\D, m_\al)).
\end{equation}
It is well known that for $1 < p < \infty$, $L_a^p(\D,m_\al)$ is reflexive. The case $p=1$ is the nonreflexive case and the duality relations have been determined as follows:
\begin{equation}\label{L1dual}
    (L_a^1(\D,m_\al))^* \approx \B_{\infty}(\D),
\end{equation}
and
\begin{equation}\label{L1predual}
    (\B_{\infty,0}(\D))^* \approx L_a^1(\D,m_\al)
\end{equation}
under the duality pairings given by, respectively
\begin{equation}\label{L1dualp}
    \langle f,g\rangle=\int_\D f(z)\overline{g(z)}\,dm_\al\quad (f\in L^1_a(\D, m_\al),\, g\in \B_{\infty}(\D))
\end{equation}
and
\begin{equation}\label{L1predualp}
    \langle f,g\rangle=\int_\D f(z)\overline{g(z)}\,dm_\al\quad (f\in \B_{\infty,0}(\D),\, g\in L^1_a(\D, m_\al)).
\end{equation}
In other words, the dual and predual spaces of the nonreflexive Bergman space $L^1_a(\D, m_\al)$ are the Bloch and Little Bloch spaces respectively.
For a comprehensive account of the theory of Bloch and Bergman spaces, we refer to \cite{Dur2, Gar, Pel, Zhu2, Zhu}.\\
In \cite{BBMM}, all the self analytic maps $(\vj_t)_{t\geq 0} \subseteq \Aut (\uP)$ of the upper half plane $\uP$ were identified and classified according to the location of their fixed points into three distinct classes, namely: scaling, translation and rotation groups. For each self analytic map $\vj_t$, we define a corresponding group of weighted composition operator on $\Hol(\uP)$ by
\begin{equation} S_{\vj_t} f(z)\;=\;(\vj_t'(z))^\g f(\vj_t(z)),\end{equation} for some appropriate weight $\g$.\\
 It is noted in \cite[Section 5]{BBMM} that for the rotation group, we consider the corresponding group of weighted composition operators defined on the analytic spaces of the disc $\Hol(\D)$ given by
\begin{equation}\label{eq2}
    T_tf(z) = e^{ict}f(e^{ikt}z)\quad \mbox{ with } c,k\in \mathbb{R},\; k\neq 0.
\end{equation}
The study of composition operators on spaces of analytic functions still remains an active area of research. For Bloch spaces, most studies have only focussed on the boundedness and compactness of these operators. See for instance \cite{BMM, Fu, Ohno, Shi, Wulan}.
In \cite{BBMM} and \cite{Bon}, both the semigroup and spectral properties of the group $(T_t)_{t\in \mathbb{R}}$ were studied in detail on the Hardy and Bergman spaces. The aim of this paper is to extend the analysis of the group $(T_t)_{t\in \mathbb{R}}$ from the Hardy and Bergman spaces to the setting of the Little Bloch space. Specifically, we apply the theory of semigroups as well as spectral theory of linear operators on Banach spaces to study the properties of the group of weighted composition operators given by equation \eqref{eq2} on the little Bloch space of the disk. As a consequence, we shall complete the analysis of the adjoint group on the dual of the nonreflexive Bergman space $L_a^1(\D,m_\al)$. The analysis of the adjoint group on the reflexive Bergman space, that is, $L_a^p(\D,m_\al)$ for $1 < p < \infty$, was considered exhaustively in \cite{Bon}. We shall also consider a specific automorphism of $\uP$ and carry out an analysis of the corresponding composition operator.\\
If $X$ is an arbitrary Banach space, let $\mathcal{L}(X)$ denote the algebra of bounded linear operators on $X$. For a linear operator $T$ with domain $\dom(T) \subset X$, denote the spectrum and point spectrum of $T$ by $\s(T)$ and $\s_p(T)$ respectively. The resolvent set of $T$ is $\rho(T) = \C\setminus \s(T)$ while $r(T)$ denotes its spectral radius. For a good account of the theory of spectra, see \cite{Dun, Con, Neu}.
If $X$ and $Y$ are arbitrary Banach spaces and $U\in{\mathcal L}(X,Y)$ is an invertible operator, then clearly  $(A_t)_{t\in \R}\subset {\mathcal L}(X)$ is a strongly continuous group if and only if $B_t:=UA_t U^{-1}$, $t\in\R$, is a strongly continuous group in ${\mathcal L}(Y)$. In this case, if $(A_t)_{t\in \R}$ has generator $\Gamma$, then $(B_t)_{t\in \R}$ has generator $\Delta=U\Gamma U^{-1}$ with domain
$\dom(\Delta)=U\dom(\Gamma) :=\left\{y \in Y : Uy \in \dom(\G)\right\}.$
Moreover, $\s_p(\Delta)=\s_p(\Gamma),$ and $\s(\Delta)=\s(\Gamma),$ since if $\lambda$ is in the resolvent set $\rho(\Gamma):=\C\setminus\s(\Gamma)$, we have
that $R(\lambda,\Delta)=UR(\lambda,\Gamma)U^{-1}$. See for example \cite[Chapter II]{Eng} and \cite[Chapter 3]{Neu}.

\section{Groups of Composition operators on the Little Bloch space}
We consider the group of weighted composition operators $(T_t)_{t\in \R}$ given by equation \eqref{eq2} and defined on the little Bloch space $\B_{\infty,0}(\D)$ as $T_tf(z)=e^{ict}f(e^{ikt}z)$ where $c,k\in \R$, $k\neq 0$ and $\forall \,f\in \B_{\infty,0}(\D)$. We denote the infinitesimal generator of the group $(T_t)_{t\in \R}$ by $\G_{c,k}$ and give some of its properties in the following Proposition,
\begin{prop}\label{gen1}
\begin{enumerate}
  \item $(T_t)_{t\in \R}$ is a strongly continuous group of isometries on $\B_{\infty,0}(\D)$.
  \item The infinitesimal generator $\G_{c,k}$ of $(T_t)_{t\in \R}$ on $\B_{\infty,0}(\D)$ is given by $\G_{c,k}f(z)= i\left(cf(z)+kzf'(z)\right)$ with domain $\dom(\G_{c,k}) = \left\{f\in \B_{\infty,0}(\D): zf' \in \B_{\infty,0}(\D)\right\}$.
\end{enumerate}
\end{prop}
\begin{proof}
To prove isometry, we have
\begin{eqnarray*}
% \nonumber to remove numbering (before each equation)
  \|T_tf\|_{\B_{\infty}(\D)} &=& |T_t f(0)|+ \sup_{z\in \D}\left( 1-|z|^2 \right)\left|(T_t f)'(z) \right| \\
   &=& |e^{ict}f(0)|+ \sup_{z\in \D}\left( 1-|z|^2 \right)\left|e^{ict}e^{ikt}f'(e^{ikt}z) \right|  \\
   &=& |f(0)|+ \sup_{z\in \D}\left( 1-|z|^2 \right)\left|f'(e^{ikt}z) \right|.
\end{eqnarray*}
By change of variables, let $\w=e^{ikt}z$. Then
\begin{eqnarray*}
% \nonumber to remove numbering (before each equation)
  \|T_tf\|_{\B_{\infty}(\D)} &=& |f(0)|+ \sup_{\w\in \D}\left( 1-|\w|^2 \right)\left|f'(\w) \right| \\
   &=&  \|f\|_{\B_{\infty}(\D)},\quad \mbox{ as desired.}
\end{eqnarray*}
To prove strong continuity, we shall use the density of polynomials in $\B_{\infty,0}(\D)$. Therefore it suffices to show that for $\left(z^n\right)_{n\geq 0}$; \[\lim_{t\to 0^+}\|T_t z^n - z^n\|_{\B_{\infty,0}(\D)} = 0.\]
Now, $T_tz^n - z^n = e^{ict}\left(e^{ikt}z\right)^n - z^n = \left(e^{i(c+kn)t}-1\right)z^n$. Therefore,
\begin{eqnarray*}
% \nonumber to remove numbering (before each equation)
  \lim_{t\to 0^+}\|T_tz^n - z^n\|_{\B_{\infty,0}(\D)}&=& \lim_{t\to 0^+} \left(\sup_{z\in \D}\left(1-|z|^2\right)\left| (T_t z^n - z^n)'\right|\right) \\
   &=& \lim_{t\to 0^+} \left(\sup_{z\in \D}\left(1-|z|^2\right)\left|n \left(e^{i(k+kn)t} - 1\right)z^{n-1}\right|\right) \\
   &=& 0,\quad \mbox{ as claimed.}
\end{eqnarray*}
Now, for the infinitesimal generator $\G_{c,k}$, let $f\in \dom(\G_{c,k})$ in $\B_{\infty,0}(\D)$, then the growth condition \eqref{eq1} implies that
\begin{eqnarray*}
% \nonumber to remove numbering (before each equation)
  \G_{c,k}f(z) &=& \lim_{t\to 0^+}\frac{e^{ict}f(e^{ikt}z)-f(z)}{t} = \frac{\partial}{\partial t}\left.\left(e^{ict}f(e^{ikt}z)\right)\right|_{t=0} \\
   &=& i(cf(z)-izf'(z)).
\end{eqnarray*}
Therefore $\dom(\G_{c,k}) \subseteq \left\{f\in \B_{\infty,0}(\D): zf' \in \B_{\infty,0}(\D)\right\}$.
Conversely, if $f\in \B_{\infty,0}(\D)$ is such that $zf'\in \B_{\infty,0}(\D)$, then $F(z) = i(cf(z)+kzf'(z)) \in \B_{\infty,0}(\D)$ and for all $t>0$,
 \begin{align*}
 \frac{T_{t}f(z)-f(z)}{t}&=\frac1{t}\int_{0}^{t}\frac{\partial}{\partial s}\left(T_{s}f(z)\right)\,ds\\
 &=\frac1{t}\int_{0}^{t}e^{ics}\left(i(c f(e^{iks} z)+k(e^{iks}z)f'(e^{iks}z))\right)\,ds\\
&=\frac1{t}\int_{0}^{t} T_sF(z)\,ds.
\end{align*}
Strong continuity of $(T_s)_{s\ge0}$ implies that $$\left\|\frac1{t}\int_{0}^{t} T_sF\,ds- F\right\|\leq \frac1{t}\int_{0}^{t}\left\| T_sF- F\right\|\,ds \to 0 \mbox{ as } t\to 0^+.$$
Thus,
$\dom(\Gamma_{c,k})\supseteq \left\{f\in \B_{\infty,0}(\D): zf' \in \B_{\infty,0}(\D)\right\}$.
\end{proof}

Define $M_z$, $Q$ on $\Hol(\D)$ by $M_zf(z) = zf(z)$ and $Qf(z) = \frac{f(z)-f(0)}{z}$, $\left(Qf(0)=f'(0)\right)$. More generally, $Q^mf(z) = \sum_{k=m}^{\infty} \frac{f^{(k)}(0)}{k!}z^{k-m}$, $Q^mf(0) = \frac{f^{m}(0)}{m!}$. Then $M_z^mQ^m f = \sum_{m}^{\infty} \frac{f^{(k)}(0)}{k!}z^k$ and $Q^mM_z^m f =f$. We now give the following proposition;
\begin{prop}\label{multprop}
\begin{enumerate}
  \item $M_z: \B_{\infty}(\D)\,\to\,\B_{\infty}(\D)$ is bounded
  \item $M_z\B_{\infty,0}(\D) \subseteq \B_{\infty,0}(\D)$
  \item $Q: \B_{\infty,0}(\D)\,\to\,\B_{\infty,0}(\D)$ is bounded
  \item For $m\geq 1$, $M_z^m\B_{\infty,0}(\D) = \left\{f\in \B_{\infty,0}(\D): f^k(0)=0 \,\forall\,\,k<m \right\}.$ In particular, $M_z\B_{\infty,0}(\D)$ is closed in $\B_{\infty,0}(\D)$.
\end{enumerate}
\end{prop}
\begin{proof}
If $f\in \B_{\infty}(\D)$, then for all $z\in \D$,
\begin{eqnarray*}
% \nonumber to remove numbering (before each equation)
  (1-|z|^2)|(zf)'| &=& (1-|z|^2)|zf'(z) + f(z)| \\
   &\leq& (1-|z|^2)|f'(z)| + (1-|z|^2)|f(z)| \\
   &\leq&  (1-|z|^2)|f'(z)| + (1-|z|^2)\left(1+\tfrac1{2}\log \left( \frac{1+|z|}{1-|z|}\right)\right)\|f\|_{\B_{\infty}(\D)}.
\end{eqnarray*}
Therefore assertions (1) and (2) follow. For (3), if $f\in \B_{\infty,0}(\D)$, then for $|z|<1$,
\begin{eqnarray*}
% \nonumber to remove numbering (before each equation)
  (1-|z|^2)|(Qf)'(z)| &=& (1-|z|^2)\left|\frac{zf'(z)-f(z)+f(0)}{z^2} \right| \\
   &\leq& \frac{(1-|z|^2|f'(z)|)}{|z|} \\
   && + \frac{(1-|z|^2)\left(1+\tfrac1{2}\log \left( \frac{1+|z|}{1-|z|}\right)\right)\|f\|_{\B_{\infty}(\D)}}{|z|^2}\\
   && +\frac{(1-|z|^2)\|f\|_{\B_{\infty}(\D)}}{|z|^2}\,\to\,0 \mbox{ as } |z|\to 1.
\end{eqnarray*}
Thus $Qf \in \B_{\infty,0}(\D)$. To prove (4), let $f\in \B_{\infty,0}(\D)$ and $f(0)=0.$ Then $f=M_zQf \in M_z\B_{\infty,0}(\D)$. The reverse inclusion is obvious. Therefore, the one-to-one and onto mapping $M_z: \B_{\infty,0}(\D)\,\to\,\{f\in \B_{\infty,0}(\D): f(0)=0\}$ is bounded. So the open mapping theorem implies that the inverse is bounded. It therefore follows that $Q: \Span(1)\oplus M_z\B_{\infty,0}(\D)\,\to\,\B_{\infty,0}(\D)$ is bounded.
\end{proof}

\begin{prop}\label{propresolvent}
Let $\G_{c,k}$ be the infinitesimal generator of the group $(T)_{t\in \R}$ given by \eqref{eq2} on $\B_{\infty,0}(\D)$, then
\begin{enumerate}
  \item $\G_{c,k} = ic + k \G_{0,1}$ with domain $\dom(\G_{c,k}) = \dom(\G_{0,1}) = \left\{f\in \B_{\infty,0}(\D): zf' \in \B_{\infty,0}(\D)\right\}.$
  \item $\s(\G_{c,k}) = \left\{ ic + k \s(\G_{0,1}) \right\}$, and $\s_p(\G_{c,k}) = \left\{ ic + k \s_p(\G_{0,1}) \right\}$.
 \end{enumerate}
In fact, $\la \in \rho(\G_{0,1})$ if and only if $ic + k \la \in \rho(\G_{c,k})$, and
\begin{equation}
R(ic + k\la, \G_{c,k}) = \frac1{k} R(\la,\G_{0,1}).
\end{equation}
\end{prop}
\begin{proof}
See \cite[Lemma 4.3]{Bon}.
\end{proof}
As a result of Proposition \ref{propresolvent} above and without loss of generality, we restrict our attention to the generator $\G_{0,1}$ instead of $\G_{c,k}$ as the cases $c\neq 0$ and $k\neq 1$ where $k\neq 0$ can be easily obtained from $\G_{0,1}$. Indeed, $\G_{0,1}f(z) = iz f'(z)$ with domain $\dom(\G_{0,1})=\left\{f\in \B_{\infty,0}(\D): zf' \in \B_{\infty,0}(\D)\right\}$ is the infinitesimal generator of the group $T_tf(z) = f(e^{it}z)$ which is exactly the case when $c=0$ and $k=1$ in equation \eqref{eq2}. We now give the spectral properties of the generator $\G_{0,1}$ as well as the resulting resolvents in the following theorem;
\begin{thm}\label{thrmresolvent1}
\begin{enumerate}
  \item $\s(\G_{0,1}) = \s_p(\G_{0,1}) = \{in: n\in \Z_+\}$, and for each $n\ge0$,\hfil\break $\ker (in-\Gamma_{0,1})=\mbox{span}(z^n)$.
  \item If $\la\in \rho(\G_{0,1})$, then $M_z\B_{\infty,0}(\D)$ is $R(\la,\G_{0,1})$ - invariant $\forall m\in \Z_+$, $m> \Im(\la)$. Moreover, if $h\in M_z^m\B_{\infty,0}(\D)$, then
      \begin{equation}\label{resolvent1}
        R(\la,\G_{0,1}) = iz^{-\la t}\int_0^z \w^{i\la -1}h(\w)\,d\w
         = iz^m \int_0^1 t^{m+i\la-1}\left(Q^mh\right)(tz)\,dt.
      \end{equation}
  \item For $\la \in \rho(\G_{0,1})$, the resolvent operator $R(\la,\G_{0,1})$ is compact.
  \item  $\s(R(\la,\G_{0,1})) = \s_p(R(\la,\G_{0,1})) = \left\{ w\in \C: \left|w-\tfrac1{2\Re(\la)}\right|= \tfrac1{2\Re(\la)} \right\}$. Moreover,
 \[r(R(\la,\G_{0,1})) = \|R(\la,\G_{0,1})\| = \tfrac1{|\Re(\la)|}. \]
\end{enumerate}
\end{thm}
\begin{proof}
Since each $T_t$ is an invertible isometry, its spectrum satisfies $\s(T_{t})\subseteq \partial \D$, and the spectral mapping theorem for strongly continuous groups (see for example \cite[Theorem V.2.5]{Eng} or \cite{Paz}) implies that $e^{t\s(\Gamma_{0,1})}\subseteq \s(T_{t}).$ Thus, $e^{t\s(\Gamma_{0,1})}\subseteq \partial \D\;\Rightarrow\;|e^{t\s(\Gamma_{0,1})}|=1\;\Rightarrow\;e^{t\Re(\w)}=1\;\Rightarrow\;\Re(\w)=0$ for $\w \in \s(\G_{0,1})$. It immediately follows that $\s(\Gamma_{0,1})\subseteq i\R$. \\
We now solve the resolvent equation: If $\la \in \C$ and $h\in \Hol(\D)$, $(\la - \G)f=h$. This is equivalent to
\[f'(z) + \frac{i\la}{z}f(z)=\frac{i}{z}h(z),\quad (z\neq 0)\] or
\[\left(z^{i\la}f(z) = iz^{i\la-1}h(z),\quad (z\in \D\setminus (-1,0].\right)\]
In particular, $(\la -\G)f=0$ if and only $f(z)=Kz^{-i\la}$, where $K$ is a constant. Since $z^{-i\la} \in \Hol(\D)$ if and only if $-i\la \in \Z_+$, it follows that
\[\s_p(\G_{0,1}) = \left\{in: n\in \Z_+\right\}\] with $\ker(in-\G_{0,1}) = \Span(z^n)$. Moreover, if $n\in \Z_+$ and $\la \in \s_p(\G_{0,1})$, then
\[(\la - \G)f = z^n\] has a unique solution
\[f(z) = \frac{1}{\la-in} z^n.\]
Notice that for $\la \notin \s_{p}(\G_{0,1})$ and $f\in \dom(\G_{0,1})$, $(\la-\G)f(0)=\la f(0)$. More generally, if $f(z)=z^ng(z)$ with $g(0)\neq 0$, then
\begin{eqnarray*}
% \nonumber to remove numbering (before each equation)
  (\la -\G)f &=& \la f -z(z^mg)' \\
   &=& z^m\left(\la g -mz^m g - z^{m+1}g' \right).
\end{eqnarray*}
Note that the functions $(\la-\G)f$ and $f$ have the same order of zero at 0. Thus $M_z^m\B_{\infty,0}(\D)$ is invariant under $\la-\G_{0,1}$.\\
Fix $\la\in \C\setminus \s_p(\G_{0,1})$ and let $m>\Im(\la)$. If $h=z^mg$ with $g\in \B_{\infty,0}(\D)$, then
\[i\int_0^z \w^{i\la-1}h(\w)\,d\w = iz^{m+i\la}\int_0^1 t^{m+i\la-1}g(tz)\,dt.\]
Thus $(\la-\G)h$ has a unique solution
\[f(z)=iz^m \int_0^1 t^{m+i\la-1}(Q^mh)(tz)\,dt.\]
If $u\in \B_{\infty}(\D)$ and $0\leq t < 1,$ then
\begin{eqnarray}\label{eq2.1}
% \nonumber to remove numbering (before each equation)
\nonumber  \|u(tz)\|_{\B_{\infty}(\D)} &=& \sup_{|z|<1}\left(1-|z|^2\right)t|u'(tz)| \\
\nonumber   &\leq& \sup_{|z|<1}\left(1-t^2|z|^2\right)|u'(tz)| \\
   &\leq&  \|u\|_{\B_{\infty}(\D)}.
\end{eqnarray}
Thus $\|f\| \leq \frac{1}{m-\Im(\la)} \|M_z^m\|\|Q^m\|\|h\|.$ Now, $\forall\, m\geq 1,$
\begin{equation}\label{blochdecomp}
\B_{\infty,0}(\D)=\Span(z^n)\oplus M_z^m\B_{\infty,0}(\D)
\end{equation}
and \[\left.R(\la,\G_{0,1})\right|_{\Span(z^n)_{0\leq n < m}} =
  \begin{pmatrix}
\frac{1}{\la}&&&&\\
&\frac{1}{\la-i}&&\mbox{\Huge 0}&\\
&&\ddots&&\\
&&&&\\
&\mbox{\Huge 0}&&&\\
&&&&\frac{1}{\la-(m-1)i}
\end{pmatrix}
.\] Thus $\la\notin \s_p(\G_{0,1})$ implying that $R(\la,\G_{0,1})$ is bounded on $\B_{\infty,0}(\D)$. Therefore $\s(\G_{0,1})=\s_p(\G_{0,1})$. This proves (1) and (2).\\
To prove the compactness of the resolvent operator, we argue as in \cite[Theorem 5.2]{BBMM}.
Fix $\lambda\in \rho(\Gamma_{0,1})$ and let $m\in \Z_+$ be such that $\Im(\lambda)< m$. Then by equation \eqref{blochdecomp}, it suffices to show that $R_m(\lambda,\Gamma_{0,1})=\left.R(\lambda,\Gamma_{0,1})\right|_{M_z^m\B_{\infty,0}(\D)}$ is compact.

Let ${\mathcal A}(r\D)$, $r>0$, be the disc algebra ${\mathcal A}(r\D)=C(r\overline{\D})\cap\Hol(r\D)$, equipped with the supremum norm, and for each $t$, $0\le t<1$, and $f\in\Hol(\D)$, let $H_tf(z) = f_t(z) = f(tz)$. Then by equation \eqref{eq2.1}, for every $t\in [0,1)$, $H_t$ is a contraction on $\B_{\infty,0}(\D)$.

Now, by equation \eqref{resolvent1}, $R_m(\lambda,\Gamma_{0,1})= i M_z^m\int_0^1 t^{m +i\la -1} H_tQ^m\,dt$ with convergence in norm.
Define $C_r = i M_z^m\int_0^r t^{m +i\la -1} H_tQ^m\,dt$ on $M^m_z\B_{\infty,0}(\D)$ for $0<r<1$. Then $$\|R_m-C_r\|\le \int_r^1 t^{m -\Im(\lambda) -1} \|Q\|^m\,dt = \frac{\|Q\|^m}{m-\Im(\lambda)}(1-r^{m-\Im(\lambda)})\to 0$$ as $r\to1^-$.
Choosing $s$ so that $1<s<r^{-1}$, we have that $C_r: M^m_z\B_{\infty,0}(\D)\to M^m_z\B_{\infty,0}(\D)$ factors through ${\mathcal A}(s\D)$. If $\mathbb{B}$ denotes the closed unit ball of $M^m_z\B_{\infty,0}(\D)$, let $h=Q^mf$ $(f\in M^m_z\B_{\infty,0}(\D))$. Then $\forall t$, $0\leq t\leq r$, the growth condition \eqref{eq1} implies that for $|z|\leq s$,
\begin{equation*}
    |h(tz)| \leq \left(1+ \frac{1}{2}\log\left(\frac{1+rs}{1-rs}\right)\right)\|h\|_{\B_{\infty,0}(\D)}
\end{equation*}
and
\begin{equation*}
    \left| \frac{d}{dt}h(tz)\right| \leq \frac{\|h\|_{\B_{\infty,0}(\D)}}{1-rs}.
\end{equation*}
Let $K=\left(1+ \frac{1}{2}\log\left(\frac{1+rs}{1-rs}\right)\right)\|h\|_{\B_{\infty,0}(\D)}$. Thus for $|z|\leq s$,
\begin{equation*}
\left|C_rf(z)\right| \le K\frac{s^mr^{m-\Im(\lambda)}}{m-\Im(\lambda)},
\end{equation*} and
\begin{equation*}
\left|\frac{d}{dz}C_rf(z)\right| \le K\frac{ms^{m-1}r^{m-\Im(\lambda)}}{m-\Im(\lambda)} + \frac{s^{m}r^{m-\Im(\lambda)}}{m-\Im(\lambda)}\frac{\|h\|_{\B_{\infty,0}(\D)}}{1-rs}.
\end{equation*}
 Thus by Arzela-Ascoli, $C_r\mathbb{B}$ is pre-compact in ${\mathcal A}(s\D)$ which further implies that $C_r\mathbb{B}$ is pre-compact in $\B_{\infty,0}(\D)$ by the continuous embeddedness of ${\mathcal A}(s\D)$ in $\B_{\infty,0}(\D)$. Therefore each $C_r$ is compact in ${\mathcal L}(M^m_z\B_{\infty,0}(\D))$ and as a result, $R_m(\lambda,\Gamma_{0,1})=(\mbox{norm})\lim_{r\to1^-}C_r$ is compact as well.\\
 The spectral mapping theorem for resolvents as well as assertion (1) above implies that
 \begin{align*}\s(R(\la,\G_{0,1}))=\s_p(R(\la,\G_{0,1}))&= \{\tfrac1{\la-im}: m\in \Z_+\}\cup \{0\}\\
 &= \left\{\w\in \C:|\w-\tfrac1{2\Re(\la)}|=\tfrac1{2|\Re(\la)|}\right\}.\end{align*}
Clearly the spectral radius $r(R(\la,\G_{0,1}))=\frac1{|\Re(\la)|}$ and therefore by the Hille-Yosida theorem, it follows that $\frac1{|\Re(\la)|} = r(R(\la,\G_{0,1})) \le \|R(\la,\G_{0,1})\|\le \frac1{|\Re(\la)|}$, as desired.
\end{proof}
As a consequence, the properties of the general group $T_t$ given by equation \eqref{eq2} is the following
\begin{cor}\label{cor}
\begin{enumerate}
  \item $\s(\G_{c,k}) = \s_p(\G_{c,k}) = \{i(c+kn): n\in \Z_+\}$, and for each $n\ge0$, $\ker (i(c+kn)-\Gamma_{c,k})=\mbox{span}(z^n)$.
  \item If $\mu\in \rho(\G_{c,k})$, then $M_z\B_{\infty,0}(\D)$ is $R(\mu,\G_{c,k})$ -invariant $\forall m\in \Z_+$, $m> \Im\left(\frac{\mu-ic}{k}\right)$. Moreover, if $h\in M_z^m\B_{\infty,0}(\D)$, then
      \begin{eqnarray}\label{resolvent2}
 \nonumber  R(\mu,\G_{c,k}) &=& \frac{i}{k}z^{-(\frac{\mu-ic}{k}) t}\int_0^z \w^{i(\frac{\mu-ic}{k}) -1}h(\w)\,d\w\\
         &=& \frac{i}{k}z^m \int_0^1 t^{m+i(\frac{\mu-ic}{k})-1}\left(Q^mh\right)(tz)\,dt.
      \end{eqnarray}
  \item For $\mu \in \rho(\G_{c,k})$, the resolvent $R(\mu,\G_{c,k})$ is compact.
  \item $\s(R(\mu,\G_{c,k})) = \s_p(R(\mu,\G_{c,k})) = \left\{ w\in \C: \left|w-\tfrac1{2\Re(\mu)}\right|= \tfrac1{2\Re(\mu)} \right\}$.
  \item $r(R(\mu,\G_{c,k})) = \|R(\mu,\G_{c,k})\| = \tfrac1{2|\Re(\mu)|}. $
\end{enumerate}
\end{cor}
\begin{proof}
Following proposition \ref{propresolvent}, $\mu\in \rho(\G_{c,k})$ if and only if $\frac{\mu-ic}{k}\in \rho(\G_{0,1})$. The proof now follows at once from Theorem \ref{thrmresolvent1}. We omit the details.
\end{proof}

\section{Adjoint of the Composition group on the predual of nonreflexive Bergman space $L_a^1(\D,m_\al)$}
In studying the adjoint properties of the rotation group isometries given by equation \eqref{eq2} on Bergman spaces $L_a^p(\D,m_{\al})$, $1\le p < \infty$, the second author in \cite{Bon} considered the reflexive case, that is when $1<p<\infty$. This was an extension of the investigation of adjoint properties of the Ces\'aro operator in \cite{Sisk} on Hardy spaces, and later generalized to Bergman spaces in \cite{BBMM}.  For the nonreflexive Bergman space $L^1_a(\D,m_\al)$ (that is, $p=1$), the analysis of the adjoint of rotation group isometries remains open and forms the basis of this section. Specifically, we complete the analysis of the adjoint group of the group of isometries $T_tf(z) = e^{ict}f(e^{ikt}z)$ where $c,k\in \mathbb{R}$ with $k\neq 0$ and $\forall \, f\in L_a^1(\D,m_{\al})$.\\
Recall from section 1 the duality relation $(\B_{\infty,0}(\D))^* \approx L_a^1(\D,m_\al)$ under the integral pairing $\langle g,f\rangle=\int_\D g(z)\overline{f(z)}\,dm_\al\quad (g\in \B_{\infty,0}(\D),\, g\in L^1_a(\D, m_\al)).$
In particular, the predual of $L_a^1(\D,m_{\al})$ is the Little Bloch space $\B_{\infty,0}(\D)$.  
Thus, using this duality pairing, for every $g\in \B_{\infty,0}(\D)$, we have
\begin{align*}
% \nonumber to remove numbering (before each equation)
  \langle g, T_tf\rangle &= \int_{\D} g(z)\overline{e^{ic t}f(e^{ikt}z)}\,dm_{\al}(z) \\
   &= \int_{\D} e^{-ic t}g(z)\overline{f(e^{ikt}z)}(1-|z|^2)^{\al}\,dA(z).
\end{align*}
By a change of variables argument: Let $\w=e^{ikt}z$ so that $z=e^{-ikt}\w$ and
\begin{align*}
% \nonumber to remove numbering (before each equation)
  \langle g, T_tf\rangle &= \int_{\D} e^{-ic t}g(e^{-ikt}\w)\overline{f(\w)}(1-|e^{-ikt}\w|^2)^{\al}\,dA(\w) \\
   &= \int_{\D} e^{-ic t}g(e^{-ikt}\w)\overline{f(\w)}\,dm_{\al}(\w)\\
   &= \int_{\D} T_{-t}g(\w)\overline{f(\w)}\,dm_{\al}(\w) = \langle T_tg,f\rangle,
\end{align*}
where $T_{-t}g(\w) = e^{-ic t}g(e^{-ikt}\w)$ for all $g\in \B_{\infty,0}(\D)$. Thus, the adjoint group $T_t^*$ of $T_t$ for $t\in \R$ is therefore given by
\begin{equation}\label{predualgrp}
    T^*_tg(\w):= T_{-t}g(\w) = e^{-ic t}g(e^{-ikt}\w),\quad \mbox{ for all } g\in \B_{\infty,0}(\D).
\end{equation}
%Clearly $(S_t)_{t\in \R}$ is the adjoint group corresponding to the group $(T_t)_{t\in \R}$ on $L_a^1(\D,m_{\al})$. 
Let $\G$ denotes the infinitesimal generator of the adjoint group $T_t^*$. Using the results of Section 2, we easily obtain the properties of the group $(T^*_t)_{t\in \R}$ as we give in the following theorem;
\begin{thm}
Let $(T^*_t)_{t\in\R} \subseteq \mathcal{L}(\B_{\infty,0}(\D))$ be the adjoint group of the group of weighted composition operators $(T_t)_{t\in \R} \subseteq \mathcal{L}(L_a^1(\D,m_\al))$ given by \eqref{predualgrp}. Then the following hold:
\begin{enumerate}
  \item $(T^*_t)_{t\in\R}$ is strongly continuous group of isometries on $\B_{\infty,0}(\D)$.
  \item The infinitesimal generator $\G$ of $(T^*_t)_{t\ge 0}$ is given by $\G g(\w) = -i\left(c g(\w)+k\w g'(\w)\right)$ with domain $\dom(\G) = \left\{g\in \B_{\infty,0}(\D): \w g'\in \B_{\infty,0}(\D)\right\}$.
  \item $\s(\G) = \s_p(\G)=\left\{-i(c+kn): n\in \Z_+\right\}$, and for each $n\ge 0$, $\ker\left(-i(c+kn)-\G\right) = \Span(\w^n)$
  \item If $\mu\in \rho(\G)$, then $M_\w\B_{\infty,0}(\D)$ is $R(\mu,\G)$ -invariant $\forall m\in \Z_+$, $m> \Im\left(\frac{-\mu-ic}{k}\right)$. Moreover, if $h\in M_\w^m\B_{\infty,0}(\D)$, then
      \begin{eqnarray*}
  R(\mu,\G) &=& -\frac{i}{k}\w^{\left(\frac{\mu+ic}{k}\right) t}\int_0^\w \z^{-i\left(\frac{\mu+ic}{k}\right) -1}h(z)\,dz\\
         &=& -\frac{i}{k}\w^m \int_0^1 t^{m-i\left(\frac{\mu+ic}{k}\right)-1}\left(Q^mh\right)(t\w)\,dt.
      \end{eqnarray*}
  \item $\s(R(\mu,\G)) = \s_p(R(\mu,\G)) = \left\{ w\in \C: \left|w-\tfrac1{2\Re(\mu)}\right|= \tfrac1{2\Re(\mu)} \right\}$.
  \item $r(R(\mu,\G)) = \|R(\mu,\G)\| = \tfrac1{|\Re(\mu)|}. $
\end{enumerate}
\end{thm}
\begin{proof}
The proof follows immediately by replacing $c$ and $k$ with $-c$ and $-k$ respectively in Proposition \ref{propresolvent} and Corollary \ref{cor}. We omit the details.
%Take $c=-\g$ and $k=-1$ and define $\G:=\G_{-\g,-1}$. The result then follows immediately from Theorem \ref{thrmresolvent1}.
\end{proof}
\section{Specific Automorphism of the half-plane}
In this section, we consider a specific automorphism group $(\vj_t)_{t\in \R} \subset \Aut(\uP)$ corresponding to the rotation group given by
\begin{equation}\label{eq3.1}
 \vj_t(z) =  \frac{z\cos t - \sin t}{z \sin t + \cos t}.
\end{equation}
It can be easily verified that $\vj_t(z) = \psi\circ u_t \circ \psi^{-1}(z)$, where $u_t(z) = e^{-2it}z$. The associated group of weighted composition operators on $\Hol(\D)$ is given by $S_{\vj_t}$ and by chain rule, it follows that $S_{\vj_t} = S_{\psi^{-1}}S_{u_t}S_{\psi}$, where $S_{\psi^{-1}} = S_{\psi}^{-1}$.\\
Now, for $f \in \B_{\infty,0}(\D)$,
\begin{align*}
S_{u_t} f(z) &= \left(u_t'(z)\right)^{\g}f(u_t(z)) \\
&= e^{-2i\g t} f(e^{-2it}z).
\end{align*}
Apparently, $S_{u_t}$ can be obtained as a special case of the group $(T_t)_{t\geq 0}$ given by equation \eqref{eq2} when $c=-2\g$ and $k=-2$. Let $\G=\G_{-2\g,-2}$ be the infinitesimal generator of the group $S_{u_t}$, then the properties of $\G$ can be summarized by the following proposition;
\begin{prop}\label{prop4.1}
Let $\G$ be the infinitesimal generator of the group of isometries $S_{u_t}$ on $\B_{\infty,0}(\D)$. Then
\begin{enumerate}
  \item $\G f(z) = i\left( -2\g f(z) - 2z f'(z) \right)$ for every $f\in \B_{\infty,0}(\D)$, with domain\\
  $ \dom(\G) = \left\{ f\in \B_{\infty,0}(\D): f' \in \B_{\infty,0}(\D) \right\}$.
  \item $\s(\G) = \s_p(\G) = \{ -2(\g+n)i: n\in \Z_{+} \}$, and for each $n\geq0$,\\
  $\ker\left(-2(\g+n)i - \G\right) = \mbox{span}(z^n)$
  \item If $\mu\in \rho(\G)$, then $\ran(M^m_z)$ is $R(\mu,\G)$-invariant for every $m\in\Z_{+}$, $m > \Im\left( -(\mu+2\g i)/2 \right)$. Moreover, if $h\in \ran(M^m_z)$, then
      \[R(\mu,\G)h(z) = -\frac{i}{2} z^{(\frac{\mu+2i\g}{2})i}\int_0^z \w^{-(\frac{\mu-2i\g}{2})i-1}h(\w)\,d\w:= R_{\mu}h(z). \]
  %\item For $1<p<\infty$, if $\mu \in \rho(\G)$, then $R_{\mu}^* = -R_{-\bar\mu}$.
\end{enumerate}
\end{prop}
\begin{proof}
Take $c=-2\g$ and $k=-2$ in Proposition \ref{gen1} and Corollary \ref{cor}. The proof follows immediately.
\end{proof}
Now, using the similarity theory of semigroups, we detail the properties of the group of weighted composition operators associated with the automorphism group $(\vj_t)_{t\geq 0}$ given by \eqref{eq3.1} in the following theorem;

\begin{thm}\label{propex3}
Let $\vj_t \in \Aut(\uP)$ be given by $\vj_t(z) = \frac{z\cos t - \sin t}{z \sin t + \cos t}$, for all $t\in \R,\,z\in\uP,$ and let $S_{\vj_t}f(z): = (\vj'_t)^\g f(\vj_t(z))$ be the corresponding group of isometries on $\B_{\infty,0}(\D)$. Then
\begin{enumerate}
  \item The infinitesimal generator $\Delta$ of the group $S_{\vj_t}$ on $\B_{\infty,0}(\D)$ is given by
  \[\Delta (h(z)) = -2\g z h(z) - (1+z^2) h'(z),\] with domain $\dom(\Delta)=\{h\in \B_{\infty,0}(\D): 2\g (\w+i)h+(\w+i)^2h'\in \B_{\infty,0}(\D)\}.$
  \item $ \s_p(\Delta) = \s(\Delta) = \{-2(\g+n)i: n\in \mathbb{Z}_{+}\}$, and for each $n\ge0$, $\ker (-2(\g+n)i-\Delta)=\mbox{span}(S_\psi^{-1} z^n)$.
  \item If $\mu\in\rho(\Delta)$ and if $m\in\Z_+$ is such that $m > \Im(-\mu/2-i\g)$. Then, if $h\in \ran(M_z^m)$, we have
 \begin{multline} R(\mu,\Delta)h(z)= (z-i)^{\frac{\mu+2i\g}{2}i} (z+i)^{-(\frac{\mu+2i\g}{2}i+2\g)}\int_0^z (\w-i)^{-(\frac{\mu+2i\g}{2})i-1} \\ (\w+i)^{\frac{\mu+2i\g}{2}i+2\g-1}h(\w)\,d\w.\end{multline}
 \item $R(\mu,\Delta)$ is compact on $\B_{\infty,0}(\D)$.
 \item $\s(R(\mu,\Delta)) = \s_p(R(\mu,\Delta)) = \left\{ w\in \C: \left|w-\tfrac1{2\Re(\mu)}\right|= \tfrac1{|\Re(\mu)|} \right\}$. Morover,
 \[r(R(\mu,\Delta)) = \|R(\mu,\Delta)\| = \tfrac1{2\Re(\mu)}. \]
\end{enumerate}
\end{thm}
\begin{proof}
Since $\vj_t(z) = \psi\circ u_t \circ \psi^{-1}(z)$, it follows that $S_{\vj_t} = S_{\psi^{-1}}S_{u_t}S_{\psi} = S_{\psi}^{-1}S_{u_t}S_{\psi}$, where $S_{\psi}$ is invertible. Let $\Delta$ be the generator of $S_{\vj_t}$ and $\G:=\G_{-2\g,-2}$ be the generator of $S_{u_t}$. Then using similarity theory as presented in section 1 of this paper, we have that:
\begin{itemize}
  \item [(a)] $\Delta = S_g \G S_g^{-1} \mbox{ with domain } \dom(\Delta) = S_g \dom(\G)$
  \item [(b)] $\s(\Delta) = \s(\G)$ and $\s_p(\Delta) = \s_p(\G)$
  \item [(c)] If $\mu \in \rho(\Delta)$, then $R(\mu,\Delta) = S_{\psi}^{-1} R(\mu,\G) S_{\psi}$.
\end{itemize}
With relations (a)-(c) above, and using Proposition \ref{prop4.1}, a direct computation yields assertions 1-3. We omit the details and instead refer to \cite[Theorem 4.4]{Bon} for a similar computation. Assertion 4 follows from the compactness of $R(\mu,\G)$, while assertion 5 is immediate from Corollary \ref{cor}(4) as well as the Hille - Yosida theorem.
\end{proof}

%
%
%%% ------------------------------------------------------------------------

% ------------------------------------------------------------------------

\begin{thebibliography}{1}
\bibitem{Sisk} A. G. Arvanitidis, A. G. Siskakis, \textit{Ces\`aro operators on the Hardy spaces of the half plane.} Canadian Math. Bull. \textbf{56} (2013), 229--240.
\bibitem{BBMM} S. Ballamoole, J. O. Bonyo, T. L. Miller, V. G. Miller, \textit{Ces\`aro operators on the Hardy and Bergman spaces of the half plane.} Complex Anal. Oper. Theory \textbf{10} (2016), 187--203.
\bibitem{BMM} S. Ballamoole, T. L. Miller, V. G. Miller, \textit{Extensions of spaces of analytic functions via pointwise limits of bounded sequences and two integral operators on generalized Bloch spaces} Arch. Math. \textbf{101} (2013), 269--283.    
\bibitem{Bon} J. O. Bonyo, \textit{ Spectral Analysis of certain groups of isometries on Hardy and Bergman spaces.} J. Math. Anal. Appl. \textbf{456} (2017), 1470--1481.
\bibitem{Con} J. B. Conway, \textit{A course in functional analysis.} Springer - Verlag, New York, 1985.
\bibitem{Dun} N. Dunford, J. T. Schwartz, \textit{Linear Operators Part I.} Interscience Publishers, New York, 1958.
\bibitem{Dur2} P. Duren, A. Schuster, \textit{Bergman spaces.} Mathematical Surveys and Monographs \textbf{100}, Amer. Math. Soc., Providence, RI, 2004.
\bibitem{Eng} K.-J. Engel, R. Nagel, \textit{A short course on operator semigroups.} Universitext, Springer, New York, 2006.
\bibitem{Fu} X. Fu, J. Zhang, \textit{Bloch - type spaces on the upper half - plane } Bull. Korean Math. Soc. \textbf{54} (2017), 1337--1346.
\bibitem{Gar} J. B. Garnett, \textit{Bounded Analytic Functions.} Graduate Texts in Mathematics, Revised First Edition, Springer, Berlin, 2010.
\bibitem{Neu} K. B. Laursen, M. M. Neumann, \textit{An introduction to local spectral theory.} Clarendon Press, Oxford, 2000.
\bibitem{Paz} A. Pazy, \textit{Semigroups of linear operators and applications to partial differential equations.} Applied Mathematical Sciences \textbf{40}, Springer, New York, 1983.
\bibitem{Pel} M. M. Peloso, \textit{Classical spaces of Holomorphic functions.} Technical report, Universi\`t di Milano, 2014.
\bibitem{Ohno} S. Ohno, R. Zhao, \textit{Weighted composition operators on the Bloch space} Bull. Austral. Math. Soc. \textbf{63} (2001), 177--185.
\bibitem{Shi} Y. Shi, S. Li, \textit{Differences of Composition operators on Bloch type spaces} Complex Anal. Oper. Theory \textbf{11} (2017), 227--242.
\bibitem{Wulan} H. Wulan, D. Zheng, K. Zhu, \textit{Composition operators on BMOA and the Bloch space} Proc. Amer. Math. Soc. \textbf{137} (2009), 3861--3868.
\bibitem{Zhu2} K. Zhu, \textit{ Bloch type spaces of analytic functions.} Rocky Mountain J. Math. \textbf{23} (1993), 1143--1177.
\bibitem{Zhu} K. Zhu, \textit{Operator theory in function spaces.} Mathematical Surveys and Monographs \textbf{138}, Amer. Math. Soc., Providence, 2007.

\end{thebibliography}
\end{document}